\documentclass{amsart}
\usepackage[utf8]{inputenc}
\usepackage{amsmath,amsthm,amssymb,url}
\usepackage[margin = 2.5cm]{geometry}
\usepackage{tikz}
\usepackage{ytableau}
\newtheorem{theorem}{Theorem}

\newtheorem{lemma}[theorem]{Lemma}

\theoremstyle{definition}

\newtheorem{example}[theorem]{Example}

\newcommand{\rd}[1]{{\color{red} #1}}

\def\cprime{$'$}

\newcommand{\p}[1]{\mathcal #1}

\DeclareMathOperator{\RST}{RST}
\newcommand{\Z}{{\mathbb Z}}

\title{
A bijective proof of the hook-length formula for skew shapes}

\author{Matja\v z Konvalinka}%
\address{Faculty of Mathematics and Physics, University of Ljubljana, and Institute of Mathematics, Physics and Mechanics, Ljubljana, Slovenia}
\email{}%
\urladdr{http://www.fmf.uni-lj.si/~konvalinka/}
\thanks{The author acknowledges the financial support from the Slovenian Research Agency (research core funding No. P1-0294).}

\date{\today}

\begin{document}

\begin{abstract}
Recently, Naruse presented a beautiful cancellation-free hook-length formula for skew shapes. The formula involves a sum over objects called \emph{excited diagrams}, and the term corresponding to each excited diagram has hook lengths in the denominator, like the classical hook-length formula due to Frame, Robinson and Thrall.\\
In this paper, we present a simple bijection that proves an equivalent recursive version of Naruse's result, in the same way that the celebrated hook-walk proof due to Greene, Nijenhuis and Wilf gives a bijective (or probabilistic) proof of the hook-length formula for ordinary shapes.\\
In particular, we also give a new bijective proof of the classical hook-length formula, quite different from the known proofs.
\end{abstract}

\maketitle

\section{Introduction} \label{intro}

The celebrated \emph{hook-length formula} gives an elegant product expression for the number of standard Young tableaux (all definitions are given in Section \ref{def}):

$$f^\lambda = \frac{|\lambda|!}{\prod_{u \in [\lambda]} h(u)}.$$

\medskip

The formula also gives dimensions of irreducible representations of the symmetric group, and is a fundamental result in algebraic combinatorics.  The formula was discovered by Frame,
Robinson and Thrall in~\cite{FRT} based on earlier results of Young \cite{You}, Frobenius \cite{Fro} and Thrall \cite{Thr}.  Since then, it has been reproved, generalized and extended in several different ways, and applied in a number of fields ranging from algebraic geometry to probability, and from group theory to the analysis of algorithms.

\medskip

In an important development, Greene, Nijenhuis and Wilf introduced the \emph{hook walk}, which proves a recursive version of the hook-length formula by a combination of a probabilistic and a short induction argument~\cite{GNW}, see also \cite{GNW2}. Zeilberger converted this hook-walk proof into a bijective proof~\cite{Zei}. With time, several variations of the hook walk have been discovered, most notably the $q$-version of Kerov~\cite{Ker1}, and its further generalizations and variations (see~\cite{GH,Ker2}). In \cite{ckp}, a weighted version of the identity is given, with a natural bijective proof in the spirit of the hook-walk proof. Also of note are the bijective proofs of Franzblau and Zeilberger \cite{fz} and Novelli, Pak and Stoyanovskii \cite{nps}. See also \cite{sagan}, \cite{fischer}, \cite{kon} for some proofs of the hook-length formula for shifted tableaux. There are also a great number of proofs of the more general hook-content formula due to Stanley (see e.g.~\cite[Corollary 7.21.4]{ec2}), see for example \cite{rw, krat2, krat3}.


\medskip

There is no (known) product formula for the number of standard Young tableaux of a skew shape, even though some formulas have been known for a long time. For example, \cite[Corollary 7.16.3]{ec2} gives a determinantal formula; we can compute the numbers via Littlewood-Richardson coefficients with the formula
$$f^{\lambda/\mu} = \sum_\nu c_{\mu,\nu}^\lambda f^\nu$$
and there is also a beautiful formula due to Okounkov and Olshanski \cite{oo}. The formula states that
\begin{equation}\label{oo}
f^{\lambda/\mu} = \frac{|\lambda/\mu|!}{\prod_{u \in [\lambda]} h(u)} \sum_{T \in \RST(\mu,\ell(\lambda))} \prod_{u \in [\mu]} (\lambda_{T(u)} - c(u)),
\end{equation}
where $\RST(\mu,\ell)$ is the set of \emph{reverse semistandard tableaux} of shape $\mu$, tableaux with entries $1,\ldots,\ell$ with weakly decreasing rows and strictly decreasing columns, and $c(u) = j -i$ is the \emph{content} of the cell $u = (i,j)$. See also \cite[\S 10.3]{mpp}.

\medskip

In 2014, Hiroshi Naruse \cite{naru} presented and outlined a proof of a remarkable cancellation-free generalization for skew shapes, somewhat similar in spirit to Okounkov-Olshanski's.

\medskip

An \emph{excited move} means that we move a cell of a diagram diagonally (right and down), provided that the cells to the right, below and diagonally down-right are not in the diagram. Let $\p E(\lambda/\mu)$ denote the set of all \emph{excited diagrams} of shape $\lambda/\mu$, diagrams in $[\lambda]$ obtained by taking the diagram of $\mu$ and performing series of excited moves in all possible ways. They were introduced by Ikeda and Naruse \cite{IN1}.

\medskip

Naruse's formula says that

\begin{equation} \label{eq}
 f^{\lambda/\mu} = |\lambda/\mu|! \sum_{D \in \p E(\lambda/\mu)} \prod_{u \in [\lambda] \setminus D} \frac 1{h(u)},
\end{equation}
where all the hook lengths are evaluated in $[\lambda]$.

\medskip

In \cite{mpp}, Morales, Pak and Panova give two different $q$-analogues of Naruse's formula: for the skew Schur functions, and for counting reverse plane partitions of skew shapes. The proofs of the former employ a combination of algebraic and bijective arguments, using the factorial Schur functions and the Hillman-Grassl correspondence. The proof of the latter uses the Hillman-Grassl correspondence and is completely combinatorial. See also \cite{mpp2}.

\medskip

The purpose of this paper is to give a bijective proof of an equivalent, recursive version of Naruse's result, in the same way that the hook walk gives a bijective (or probabilistic) proof of the classical hook-length formula.

\medskip

The bijection is quite easy to explain, and, in particular, gives a new bijective proof of the classical hook-length formula, rather different from the hook-walk proof or the proof due to Novelli-Pak-Stoyanovskii.

\medskip

Our main result (Theorem \ref{thm1}) is the following formula, valid for partitions $\lambda$, $\mu$ and for commutative variables $x_i$, $y_j$:
\begin{equation} \label{eqmain}
\Bigg( \sum_{\substack{\nexists i \colon \lambda_k - k \\= \mu_i - i}} x_k + \sum_{\substack{\nexists j \colon \lambda'_k - k \\= \mu'_j - j}} y_k \Bigg) \sum_{D \in \p E(\lambda/\mu)} \prod_{(i,j) \in D} (x_i + y_j) = \sum_{\mu \lessdot \nu \subseteq \lambda} \sum_{D \in \p E(\lambda/\nu)} \prod_{(i,j) \in D} (x_i + y_j).
\end{equation}

The formula specializes to the recursive version of equation \eqref{eq}. It was pointed out by Morales and Panova (personal communication) that the identity is equivalent to the identity \cite[equation (5.2)]{IN1}. See also \cite{mpp3} and Section \ref{final}.

\medskip

In Section \ref{def}, we give basic definitions and notation. In Section \ref{identity}, we motivate equation \eqref{eqmain} and show how it implies \eqref{eq}. In Section \ref{bijection}, we use a version of the bumping algorithm on tableaux to prove the identity bijectively. In Section \ref{proofs}, we present the proofs of the technical statements from Sections \ref{identity} and \ref{bijection}. We finish with some closing remarks in Section \ref{final}.

\section{Basic definitions and notation} \label{def}

A \emph{partition} is a weakly decreasing finite sequence of positive integers $\lambda = (\lambda_1,\lambda_2,\ldots,\lambda_\ell)$. We call $|\lambda| = \lambda_1 + \cdots + \lambda_\ell $ the \emph{size} of $\lambda$ and $\ell = \ell(\lambda)$ the \emph{length} of $\lambda$. We write $\lambda_i = 0$ for $i > \ell(\lambda)$. The \emph{diagram} of $\lambda$ is $[\lambda] = \{(i,j) \colon 1 \leq i \leq \ell(\lambda), 1 \leq j \leq \lambda_i \}$. We call the elements of $[\lambda]$ the \emph{cells} of $\lambda$. For partitions $\mu$ and $\lambda$, we say that $\mu$ is \emph{contained} in $\lambda$, $\mu \subseteq \lambda$, if $[\mu] \subseteq [\lambda]$. We say that $\lambda/\mu$ is a \emph{skew shape} of size $|\lambda/\mu| = |\lambda| - |\mu|$, and the diagram of $\lambda/\mu$ is $[\lambda/\mu] = [\lambda] \setminus [\mu]$. We write $\mu \lessdot \lambda$ if $\mu \subseteq \lambda$ and $|\lambda/\mu| = 1$. In this case, we also say that $\lambda$ \emph{covers} $\mu$.

\medskip

We often represent a partition $\lambda$ by its \emph{Young diagram}, in which a cell $(i,j) \in [\lambda]$ is represented by a unit square in position $(i,j)$. In this paper, we use English notation, so for example the Young diagram of the partition $\lambda = (6,5,2,2)$ is

\medskip

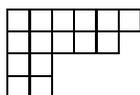
\begin{center}
\ytableausetup{smalltableaux}
\ydiagram{6,5,2,2}
\end{center}

\medskip

We often omit parentheses and commas, so we could write $\lambda = 6522$.

\medskip

A \emph{corner} of $\lambda$ is a cell that can be removed from $[\lambda]$, i.e., a cell $(i,j) \in [\lambda]$ satisfying $(i+1,j),(i,j+1) \notin [\lambda]$. An \emph{outer corner} of $\lambda$ is a cell that can be added to $[\lambda]$, i.e., a cell $(i,j) \notin [\lambda]$ satisfying $i = 1$ or $(i-1,j) \in [\lambda]$, and $j = 1$ or $(i,j-1) \in [\lambda]$. The \emph{rank} of $\lambda$ is $r(\lambda) = \max\{ i \colon \lambda_i \geq i\}$. The square $[1,r(\lambda)] \times [1,r(\lambda)]$ is called the \emph{Durfee square} of $\lambda$. The partition $6522$ has corners $(1,6)$, $(2,5)$ and $(4,2)$, outer corners $(1,7)$, $(2,6)$, $(3,3)$ and $(5,1)$, and rank $2$.

\medskip

The \emph{conjugate} of a partition $\lambda$ is the partition $\lambda'$ whose diagram is the transpose of $[\lambda]$; in other words, $\lambda'_j = \max \{ i \colon \lambda_i \geq j \}$. For example, for $\lambda = 6522$, we have $\lambda' = 442221$. The \emph{hook length} of the cell $(i,j) \in [\lambda]$ is defined by $h(i,j) = \lambda_i + \lambda'_j - i - j + 1$. For example, the hook length of the cell $(1,2) \in [6522]$ is $8$.

\medskip

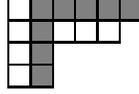
\begin{center}
\ytableausetup{smalltableaux}
 \begin{ytableau}
 {} & *(gray) & *(gray) & *(gray) & *(gray) & *(gray) \\
 & *(gray) & & &\\
 & *(gray) \\
 & *(gray) \\
\end{ytableau}
\end{center}

\medskip

The \emph{hook} of a cell $u = (i,j) \in [\lambda]$ is $H(u) = \{(i,j') \colon j \leq j' \leq \lambda_i \} \cup \{(i',j) \colon i \leq i' \leq \lambda'_j\}$. Obviously, we have $|H(u)| = h(u)$. The diagram $[\lambda]$ is the disjoint union of $H(i,i)$, $1 \leq i \leq r(\lambda)$, as illustrated by the following figure.

\medskip

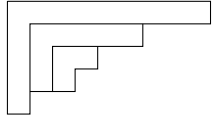
\begin{center}
\begin{tikzpicture}[scale = 0.3]
 \draw (0,0) -- (9,0) -- (9,-1) -- (1,-1) -- (1,-5) -- (0,-5) -- (0,0);
 \draw (6,-1) -- (6,-2) -- (2,-2) -- (2,-4) -- (1,-4);
 \draw (4,-2) -- (4,-3) -- (3,-3) -- (3,-4) -- (2,-4);
\end{tikzpicture}
\end{center}

\medskip

A \emph{standard Young tableau} (or \emph{SYT} for short) of shape $\lambda$ is a bijective map $T \colon [\lambda] \to \{1,\ldots,|\lambda|\}$, $(i,j) \mapsto T_{ij}$, satisfying $T_{ij} < T_{i,j+1}$ if $(i,j),(i,j+1) \in [\lambda]$ and $T_{ij} < T_{i+1,j}$ if $(i,j),(i+1,j) \in [\lambda]$. The number of SYT's of shape $\lambda$ is denoted by $f^\lambda$. The following illustrates $f^{32} = 5$.

\medskip

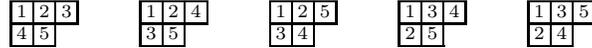
\begin{center}
\ytableausetup{centertableaux}
\ytableausetup{smalltableaux}
\begin{ytableau}
1 & 2 & 3 \\
4 & 5
\end{ytableau}
\qquad
\begin{ytableau}
1 & 2 & 4 \\
3 & 5
\end{ytableau}
\qquad
\begin{ytableau}
1 & 2 & 5 \\
3 & 4
\end{ytableau}
\qquad
\begin{ytableau}
1 & 3 & 4 \\
2 & 5
\end{ytableau}
\qquad
\begin{ytableau}
1 & 3 & 5 \\
2 & 4
\end{ytableau}
\end{center}

\medskip

The \emph{hook-length formula} gives a product expression for the number of standard Young tableaux:

$$f^\lambda = \frac{|\lambda|!}{\prod_{u \in [\lambda]} h(u)}.$$

For example, $f^{32} = \frac{5!}{4 \cdot 3 \cdot 1 \cdot 2 \cdot 1} = 5$.

\medskip

Analogously, if $\mu \subseteq \lambda$, we can define a standard Young tableau of skew shape $\lambda/\mu$ as a map $T \colon [\lambda/\mu] \to \{1,\ldots,|\lambda/\mu|\}$, $(i,j) \mapsto T_{ij}$, satisfying $T_{ij} < T_{i,j+1}$ if $(i,j),(i,j+1) \in [\lambda/\mu]$ and $T_{ij} < T_{i+1,j}$ if $(i,j),(i+1,j) \in [\lambda/\mu]$. The number of SYT's of shape $\lambda/\mu$ is denoted by $f^{\lambda/\mu}$. The following illustrates $f^{43/2} = 9$:

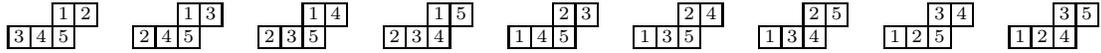
\begin{center}
\ytableausetup{centertableaux}
\ytableausetup{smalltableaux}
\begin{ytableau}
\none & \none & 1 & 2 \\
3 & 4 & 5
\end{ytableau}
\quad
\begin{ytableau}
\none & \none & 1 & 3 \\
2 & 4 & 5
\end{ytableau}
\quad
\begin{ytableau}
\none & \none & 1 & 4 \\
2 & 3 & 5
\end{ytableau}
\quad
\begin{ytableau}
\none & \none & 1 & 5 \\
2 & 3 & 4
\end{ytableau}
\quad
\begin{ytableau}
\none & \none & 2 & 3 \\
1 & 4 & 5
\end{ytableau}
\quad
\begin{ytableau}
\none & \none & 2 & 4 \\
1 & 3 & 5
\end{ytableau}
\quad
\begin{ytableau}
\none & \none & 2 & 5 \\
1 & 3 & 4
\end{ytableau}
\quad
\begin{ytableau}
\none & \none & 3 & 4 \\
1 & 2 & 5
\end{ytableau}
\quad
\begin{ytableau}
\none & \none & 3 & 5 \\
1 & 2 & 4
\end{ytableau}
\end{center}

\medskip

Suppose that $D \subseteq [\lambda]$. If $(i,j) \in D$, $(i+1,j),(i,j+1),(i+1,j+1) \in [\lambda] \setminus D$, then an \emph{excited move} with respect to $\lambda$ is the replacement of $D$ with $D' = D \setminus \{(i,j)\} \cup \{(i+1,j+1)\}$. If $\mu$ and $\lambda$ are partitions, then an \emph{excited diagram} of shape $\lambda/\mu$ is a diagram contained in $[\lambda]$ that can be obtained from $[\mu]$ with a series of excited moves. Let $\p E(\lambda/\mu)$ denote the set of all excited diagrams of shape $\lambda/\mu$. We have $\p E(\lambda/\mu) = \emptyset$ unless $\mu \subseteq \lambda$. The following shows $\p E(43/2)$.

\medskip

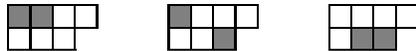
\begin{center}
\ytableausetup{smalltableaux}
 \begin{ytableau}
 *(gray) & *(gray) & & \\
  & &\\
\end{ytableau}
\qquad
\ytableausetup{smalltableaux}
 \begin{ytableau}
 *(gray) & & & \\
  & & *(gray) \\
\end{ytableau}
\qquad
\ytableausetup{smalltableaux}
 \begin{ytableau}
 {} &  & & \\
  & *(gray)&*(gray)\\
\end{ytableau}
\end{center}

\medskip

Naruse's formula says that

\begin{equation*}
 f^{\lambda/\mu} = |\lambda/\mu|! \sum_{D \in \p E(\lambda/\mu)} \prod_{u \in [\lambda] \setminus D} \frac 1{h(u)},
\end{equation*}
where all the hook lengths are evaluated in $[\lambda]$.

\medskip

For example, the formula confirms that
$$f^{43/2} = 5! \left( \frac 1{3 \cdot 1 \cdot 3 \cdot 2 \cdot 1} + \frac 1{4 \cdot 3 \cdot 1 \cdot 3 \cdot 2} + \frac 1{5 \cdot 4 \cdot 3 \cdot 1 \cdot 3} \right) = 9.$$

\section{A polynomial identity} \label{identity}

It is clear that both sides of \eqref{eq} are equal to $1$ if $\lambda = \mu$. Since the minimal entry of a standard Young tableau of shape $\lambda/\mu$ must be in an outer corner of $\mu$ which lies in $\lambda$, we have $f^{\lambda/\mu} = \sum_{\mu \lessdot \nu \subseteq \lambda} f^{\lambda/\nu}$, where $\sum_{\mu \lessdot \nu \subseteq \lambda}$ denotes the sum over all partitions $\nu$ that are contained in $\lambda$ and cover $\mu$. If we show that the right-hand side of \eqref{eq} satisfies the same recursion, we are done. Therefore the statement is equivalent to the following identity:

$$|\lambda/\mu| \sum_{D \in \p E(\lambda/\mu)} \prod_{u \in [\lambda] \setminus D} \frac 1{h(u)} = \sum_{\mu \lessdot \nu \subseteq \lambda} \sum_{D \in \p E(\lambda/\nu)} \prod_{u \in [\lambda] \setminus D} \frac 1{h(u)}.$$

After multiplying by $\prod_{u \in [\lambda]} h(u)$, we get

\begin{equation} \label{eq2}
(|\lambda| - |\mu|) \sum_{D \in \p E(\lambda/\mu)} \prod_{u \in D} h(u) = \sum_{\mu \lessdot \nu \subseteq \lambda} \sum_{D \in \p E(\lambda/\nu)} \prod_{u \in D} h(u).
\end{equation}

\begin{example}
 Take $\mu = 2$ and $\lambda = 43$. There are three excited diagrams:

 \medskip

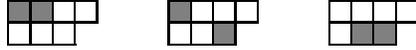
\begin{center}
\ytableausetup{smalltableaux}
 \begin{ytableau}
 *(gray) & *(gray) & & \\
  & &\\
\end{ytableau}
\qquad
\ytableausetup{smalltableaux}
 \begin{ytableau}
 *(gray) & & & \\
  & & *(gray) \\
\end{ytableau}
\qquad
\ytableausetup{smalltableaux}
 \begin{ytableau}
 {} &  & & \\
  & *(gray)&*(gray)\\
\end{ytableau}
\end{center}

 \medskip

 That means that the left-hand side of \eqref{eq2} equals
 $$(7 - 2) \left( 5 \cdot 4 + 5 \cdot 1 + 2 \cdot 1 \right) = 135.$$
 On the other hand, there are two partitions $\nu$ that cover $\mu$, and together they give three excited diagrams:

 \medskip

 \begin{center}
  \ytableausetup{smalltableaux}
 \ytableausetup{aligntableaux=top}
 \ytableaushort  {\none}  * {4,3}  * [*(gray)]{3}
\qquad\qquad
 \ytableaushort  {\none}  * {4,3}  * [*(gray)]{2,1}
  \qquad\qquad
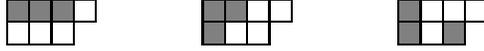
 \begin{ytableau}
 *(gray) &  &  &   \\
 *(gray)&  & *(gray)
\end{ytableau}
\end{center}

 \medskip

 That means that the right-hand side of \eqref{eq2} equals
 $$5 \cdot 4 \cdot 3 + 5 \cdot 4 \cdot 3 + 5 \cdot 3 \cdot 1 = 135.$$
\end{example}

For $i,j = 1,2,\ldots$, define
\begin{equation} \label{xy}
x_i = \lambda_i - i + \frac 1 2, \qquad y_j = \lambda'_j - j+ \frac 1 2.
\end{equation}

Clearly, for a cell $u = (i,j) \in [\lambda]$, we have $h(u) = \lambda_i + \lambda'_j - i - j + 1 = x_i + y_j$. Furthermore, since $[\lambda]$ is the disjoint union of hooks $H(i,i)$, $1 \leq i \leq r(\lambda)$, we have
\begin{equation} \label{size}
 |\lambda| = x_1 + y_1 + \cdots + x_{r(\lambda)} + y_{r(\lambda)}.
\end{equation}


\medskip

For $\lambda = 43$, we have  $x_1 = 3 \frac 1 2$, $y_1 = 1 \frac 1 2$, $x_2 = 1 \frac 1 2$, $y_2 = \frac 1 2$, and indeed $|\lambda| = x_1 + y_1 + x_2 + y_2 = 7$.

\medskip

Equation \eqref{eq2} is therefore equivalent to the following:

\begin{equation} \label{eq3}
\left( \sum_{i=1}^{r(\lambda)} (x_i + y_i) - |\mu| \right) \sum_{D \in \p E(\lambda/\mu)} \prod_{(i,j) \in D} (x_i + y_j) = \sum_{\mu \lessdot \nu \subseteq \lambda} \sum_{D \in \p E(\lambda/\nu)} \prod_{(i,j) \in D} (x_i + y_j).
\end{equation}

Note that this is \emph{not} a valid polynomial identity for every $\lambda$, $\mu$: indeed, the right-hand side is a homogeneous polynomial (of degree $|\mu|+1$), while the left-hand side is not (except when $\mu = \emptyset$ or $\mu \not\subseteq \lambda$). It represents a valid identity only for specific values of $x_i$'s and $y_i$'s.

\begin{example}
 Again, take $\mu = 2$ and $\lambda = 43$. The left-hand side of \eqref{eq3} is
 $$(x_1+y_1+x_2+y_2 - 2) \left( (x_1+y_1)(x_1+y_2) + (x_1+y_1)(x_2+y_3) + (x_2+y_2)(x_2+y_3) \right),$$
 and the right-hand side is
 $$(x_1+y_1)(x_1+y_2)(x_1+y_3) + (x_1+y_1)(x_1+y_2)(x_2+y_1) + (x_1+y_1)(x_2+y_3)(x_2+y_1).$$
 These two polynomials are \emph{not} equal, but they both specialize to $135$ when $x_1 = 3 \frac 1 2$, $y_1 = 1 \frac 1 2$, $x_2 = 1 \frac 1 2$, $y_2 = \frac 1 2$, $x_3 = - 1 \frac 1 2$, $y_3 = - \frac 1 2$. Also of note is the fact that the difference between the two polynomials is divisible by $x_2+y_2-2$.
\end{example}

However, we can replace $|\lambda| - |\mu|$ on the left-hand side of equation \eqref{eq2} with a certain \emph{homogeneous} linear polynomial (and $h(u)$ again by $x_i + y_j$ if $u = (i,j)$) and get a valid polynomial identity. This identity specializes to \eqref{eq2} for appropriate values of $x_i$'s and $y_i$'s. The motivation for the result is the following lemma, which we prove in Section \ref{proofs}. The result holds for all $\lambda$, $\mu$, even if $\mu \not\subseteq \lambda$.

\begin{lemma} \label{w}
 For arbitrary partitions $\lambda,\mu$ and $x_k = \lambda_k - k + \frac 1 2$, $y_k = \lambda'_k - k + \frac 1 2$, we have
 $$|\lambda| - |\mu| = \sum_{\substack{\nexists i \colon \lambda_k - k \\= \mu_i - i}} x_k + \sum_{\substack{\nexists j \colon \lambda'_k - k \\= \mu'_j - j}} y_k.$$
\end{lemma}

Note that while $k$, $i$ and $j$ appearing in the sums can be arbitrarily large, the summation is finite since we have $\lambda_k - k = \mu_k -k = -k$ and $\lambda'_k -k = \mu'_k -k = -k$ for large $k$.

\begin{example}
 We continue with the previous example, i.e., take $\mu = 2$ and $\lambda = 43$. We have
 $$ \begin{array}{rlrl}
 (\lambda_k-k)_{k \geq 1} =& (\underline{3},1,-3,-4,-5,\ldots), & (\mu_i-i)_{i \geq 1} = & (1,-2,-3,-4,\ldots), \\
 (\lambda'_k-k)_{k \geq 1} =& (\underline 1,0,-1,-3,-5,-6,-7,\ldots), & (\mu'_j-j)_{j \geq 1} = & (0,-1,-3,-4,-5,\ldots),
 \end{array}$$
 where elements of $(\lambda_k-k)_{k \geq 1}$ and $(\lambda'_k-k)_{k \geq 1}$ are underlined if they do not appear in $(\mu_i-i)_{i \geq 1}$ and $(\mu'_j-j)_{j \geq 1}$. Indeed, $|\lambda| - |\mu| = x_1 + y_1 = 5$.\\
 Similarly, for $\mu = 431$ and $\lambda = 765521$, we have
 $$ \begin{array}{rlrl}
 (\lambda_k-k)_{k \geq 1} =& (\underline 6,\underline 4,\underline 2,1,\underline{-3},-5,-7,-8,\ldots), & (\mu_i-i)_{i \geq 1} = & (3,1,-2,-4,-5,\ldots), \\
 (\lambda'_k-k)_{k \geq 1} =& (\underline 5,\underline 3,\underline 1,0,-1,\underline{-4},-6,-8,-9,\ldots), & (\mu'_j-j)_{j \geq 1} = & (2,0,-1,-3,-5,-6,\ldots),
 \end{array}$$
 and $|\lambda| - |\mu| = x_1 + x_2 + x_3 + x_5 + y_1 + y_2 + y_3 + y_6 = 18$.
\end{example}

The following theorem is our main result. It is a subtraction-free polynomial identity, which, by Lemma \ref{w}, specializes to equation \eqref{eq2} when $x_i = \lambda_i - i + \frac 1 2$ and $y_j = \lambda'_j - j + \frac 1 2$, and therefore implies the hook-length formula for skew diagrams.

\begin{theorem} \label{thm1}
For arbitrary partitions $\lambda$, $\mu$ and commutative variables $x_i,y_j$, we have
\begin{equation} \label{eq4}
\Bigg( \sum_{\substack{\nexists i \colon \lambda_k - k \\= \mu_i - i}} x_k + \sum_{\substack{\nexists j \colon \lambda'_k - k \\= \mu'_j - j}} y_k \Bigg) \sum_{D \in \p E(\lambda/\mu)} \prod_{(i,j) \in D} (x_i + y_j) = \sum_{\mu \lessdot \nu \subseteq \lambda} \sum_{D \in \p E(\lambda/\nu)} \prod_{(i,j) \in D} (x_i + y_j).
\end{equation}
\end{theorem}

The theorem is trivially true for $\mu \not\subseteq \lambda$, as then both sides are equal to $0$.

\begin{example}
 For $\mu = 2$ and $\lambda = 43$, we have the following identity (valid for commutative variables $x_1$, $y_1$, $x_2$, $y_2$, $x_3$, $y_3$).
 \begin{multline*}
 (x_1+y_1) \left( (x_1+y_1)(x_1+y_2) + (x_1+y_1)(x_2+y_3) + (x_2+y_2)(x_2+y_3) \right) \\
  = (x_1+y_1)(x_1+y_2)(x_1+y_3) + (x_1+y_1)(x_1+y_2)(x_2+y_1) + (x_1+y_1)(x_2+y_3)(x_2+y_1).
 \end{multline*}
 For $\mu = 431$ and $\lambda = 765521$, the first term on the left is $x_1 + x_2 + x_3 + x_5 + y_1 + y_2 + y_3 + y_6$, the second term is a sum of $14080$ monomials, and the right-hand side is a sum of $112640$ monomials.
\end{example}

The (bijective) proof of Theorem \ref{thm1} is the content of the next section.

\section{The bijection} \label{bijection}

First, we interpret the two sides of equation \eqref{eq4} in terms of certain tableaux.

\medskip

To motivate the definition, look at the following excited diagram for $\mu = 431$ and $\lambda = 765521$.

\medskip

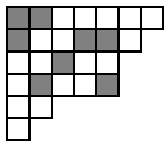
\begin{center}
   \ytableausetup{smalltableaux}
 \ytableausetup{aligntableaux=top}
 \begin{ytableau}
 *(gray) & *(gray) &  &  &  &  & \\
 *(gray)&  &  & *(gray) & *(gray) &\\
 & & *(gray) & & \\
 & *(gray) & & & *(gray) \\
 & \\
 \\
 \end{ytableau}
\end{center}

\medskip

Instead of actually moving the cells of $\mu$, write an integer in a cell of $\mu$ that indicates how many times it moves (diagonally) from the original position. For the above example, we get the following tableau of shape $\mu = 431$.

\medskip

\begin{center}
   \ytableausetup{boxsize=1.25em}
 \ytableausetup{aligntableaux=top}
 \ytableaushort
{0011,012,1}
\end{center}

\medskip

It is easy to see that the (non-negative integer) entries of the resulting tableau are weakly increasing along rows and columns (in other words, that the tableau is a \emph{reverse plane partition}): for example, if one cell is to the left of another, we cannot make an excited move on it until we make an excited move on its right neighbor. Also, every tableau with non-negative integer entries and weakly increasing rows and columns corresponds to a valid excited diagram, provided that the entry $r$ in row $i$ and column $j$ satisfies
\begin{equation} \label{crit}
j + r \leq \lambda_{i+r}.
\end{equation}
Furthermore, it is enough to check this inequality only for the corners of $\mu$. See also \emph{flagged tableaux} in \cite[\S 3.2]{mpp}.

\medskip

The contribution  $\prod_{(i,j) \in D} (x_i + y_j)$ of an excited diagram $D$ can be written as
$$\prod_{(i,j) \in [\mu]} (x_{i+T_{ij}} + y_{j+T_{ij}}),$$
where $T$ is the corresponding tableau of shape $\mu$ with non-negative integer entries and weakly increasing rows and columns. To extract the monomials from the product, choose either $x_{i+T_{ij}}$ or $y_{j+T_{ij}}$ for each $(i,j) \in [\mu]$. Write the number $T_{ij}$ in position $(i,j)$ in black if we choose $x_{i+T_{ij}}$, and in red if we choose $y_{j+T_{ij}}$. Call a tableau with non-negative integer black or red entries and weakly increasing rows and columns a \emph{bicolored tableau}. Denote by $\p B(\mu)$ the (infinite unless $\mu = \emptyset$) set of bicolored tableaux of shape $\mu$, and denote by $\p B(\mu,\lambda)$ the (finite) set of bicolored tableaux $T$ of shape $\mu$ that satisfy $j + T_{ij} \leq \lambda_{i+T_{ij}}$ for all $(i,j) \in [\mu]$.

\medskip

The weight of a bicolored tableau $T$ of shape $\mu$ is
$$w(T) = \prod_{(i,j) \in b(T)} x_{i + T_{ij}} \prod_{(i,j) \in [T] \setminus b(T)} y_{j + T_{ij}},$$
where $b(T)$ is the set of cells containing black entries of $T$.

\begin{example}
 The following are some bicolored tableaux in $\p B(431)$. A bicolored tableau is in $\p B(431,765521)$ if and only if $T_{14} \leq 1$, $T_{23} \leq 2$, $T_{31} \leq 1$, so the first three are in $\p B(431,765521)$ and the last one is not.

 \medskip

 \begin{center}
   \ytableausetup{boxsize=1.25em}
 \ytableausetup{aligntableaux=top}
 \ytableaushort
{0{\rd 0}00,01{\rd 1},0}
\qquad \qquad \ytableaushort
{{\rd 0}0 {\rd 1}1,0{\rd 2}2,{\rd 1}}
\qquad \qquad \ytableaushort
{{\rd 1}{\rd 1}{\rd 1}1,{\rd 1}{\rd 2}2,1}
\qquad \qquad \ytableaushort
{01{\rd 1}1,2{\rd 2}2,{\rd 2}}
\end{center}

\medskip

\noindent
The weights of these tableaux are $x_1^3x_2y_2x_3^2y_4$, $x_1y_1x_2^2y_2x_4y_4^2$, $x_2y_2^2y_3x_4^2y_4^2$, and $x_1x_2^2y_3x_4^2y_4^2$, respectively.
\end{example}

\medskip

We are ready to interpret both sides of equation \eqref{eq4}. The left-hand side is the enumerator of the Cartesian product $\p B(\mu,\lambda) \times \p W(\mu,\lambda)$, where
$$\p W(\mu,\lambda) = \{x_k \colon \lambda_k - k \neq \mu_i - i \mbox{ for all } i\} \cup \{y_k \colon \lambda'_k - k \neq \mu'_j - j \mbox{ for all } j\},$$
and the pair $(T,z)$ has weight $w(T)z$. The right-hand side is the enumerator (with respect to weight $w$) of the set $\bigcup_{\nu} \p B(\nu,\lambda)$, where the union is over all partitions $\nu$ that cover $\mu$ and are contained in $\lambda$.

\medskip

In the remainder of this section, we present a weight-preserving bijection between the two sides.

\medskip

The map is a natural bumping algorithm. To describe it, we first describe the \emph{insertion process}: the process of inserting a variable $z \in \{x_1,y_1,x_2,y_2,\ldots \}$ into a bicolored tableau $T$ of shape $\mu$.

\medskip

After some number of steps, $i$, $j$, $w$ and $S$ have certain values; in the beginning, $i = j = 0$, $w = z$ and $S = T$. If $w = x_k$, increase $j$ by $1$ (i.e., move to the next column) and find the largest possible $i$ (which can also be $\mu_j' + 1$ if $j = 1$ or $\mu'_j < \mu'_{j-1}$) so that we can replace $S_{ij}$ by a black $k - i$ in position $(i,j)$ and still have a weakly increasing column with non-negative integers (such an $i$ always exists, as we will see in Section \ref{proofs}). If, on the other hand, $w = y_k$, increase $i$ by $1$ (i.e., move to the next row) and find the largest possible $j$ (which can also be $\mu_i + 1$ if $i = 1$ or $\mu_i < \mu_{i-1}$) so that we can replace $S_{ij}$ by a red $k - j$ in position $(i,j)$ and still have a weakly increasing row with non-negative integers.

\medskip

Let $w$ denote the weight of the old $S_{ij}$ (i.e., $x_{i+S_{ij}}$ if $S_{ij}$ is black and $y_{j+S_{ij}}$ if $S_{ij}$ is red). Continue with the procedure until $(i,j)$ is an outer corner of $\mu$, and $S$ is a bicolored tableau of some shape $\nu$ which covers $\mu$. The procedure returns this final $S$, which we denote by $\psi_\mu(T,z)$.

\begin{example} \label{ex}
 Take $\mu = 431$, the bicolored tableau
$$T =     \ytableausetup{boxsize=1.25em}
 \ytableausetup{aligntableaux=center}
 \ytableaushort
{0{\rd 0}{\rd 0}1,01{\rd 1},0}$$
 and $z = y_1$. Since we are inserting a $y$-variable, we insert it into the first row. The variable $y_1$ can only be represented by a red $0$ in the first column, so we write a red $0$ in position $(1,1)$, and the variable bumped out is $x_1$ (represented by the black $0$ that was in position $(1,1)$ originally). Since this is an $x$-variable, we move to the right, and insert it into the second column. The variable $x_1$ can only be represented by a black $0$ in the first row, so we write a black $0$ in position $(1,2)$, and the variable bumped out is $y_2$ (represented by the red $0$ that was in position $(1,2)$ before). We have to insert it into the second row, either as a red $1$ in position $(2,1)$ or a red $0$ in position $(2,2)$. Of course, a red $1$ in position $(2,1)$ would give a decrease in column $1$, so we insert it in position $(2,2)$, and bump out a black $1$, representing $x_3$. We insert $x_3$ in column $3$, either as a black $2$ in row $1$ (but which makes the entry in $(3,1)$ larger than the entry in $(3,2)$) or as a black $1$ in row $2$. Thus we write a black $1$ in position $(3,2)$ and bump out the red $1$ representing $y_4$. We move to the next row: we can either write a red $3$ in position $(3,1)$ or a red $2$ in position $(3,2)$. Both are possible, so we pick the latter option. Now $(i,j) = (3,2)$ is an outer corner of $\mu$, so we terminate the insertion process. The final bicolored tableau is
 $$\psi_{431}(T,y_1) =     \ytableausetup{boxsize=1.25em}
 \ytableausetup{aligntableaux=center}
 \ytableaushort
{{\rd 0}0{\rd 0}1,0{\rd 0}1,0 {\rd 2}}.$$
Figure \ref{path} illustrates the insertion process. Two numbers in a cell mean that the number on the left is bumping the number on the right.
\end{example}

\begin{figure}[h]
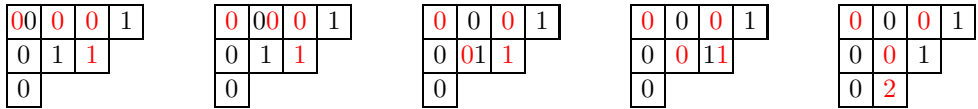

\ytableausetup{boxsize=1.25em}
 \ytableausetup{aligntableaux=center}
 \ytableaushort
{{{\rd 0}0}{\rd 0}{\rd 0}1,01{\rd 1},0} \qquad
\ytableausetup{boxsize=1.25em}
 \ytableausetup{aligntableaux=center}
 \ytableaushort
{{\rd 0}{0 \rd 0}{\rd 0}1,01{\rd 1},0} \qquad
\ytableausetup{boxsize=1.25em}
 \ytableausetup{aligntableaux=center}
 \ytableaushort
{{\rd 0}0{\rd 0}1,0{{\rd 0}1}{\rd 1},0} \qquad
\ytableausetup{boxsize=1.25em}
 \ytableausetup{aligntableaux=center}
 \ytableaushort
{{\rd 0}0{\rd 0}1,0{\rd 0}{1 \rd 1},0}\qquad
\ytableausetup{boxsize=1.25em}
 \ytableausetup{aligntableaux=center}
 \ytableaushort
{{\rd 0}0{\rd 0}1,0{\rd 0}1,0{\rd 2}}
\caption{The insertion process from Example \ref{ex}.} \label{path}
\end{figure}

\begin{theorem} \label{insert}
 The insertion process described above always terminates and is a weight-preserving bijection
 $$\psi_\mu \colon \p B(\mu) \times \{x_1,y_1,x_2,y_2,\ldots \} \longrightarrow \bigcup_{\nu} \p B(\nu),$$
 where the union is over all partitions $\nu$ which cover $\mu$.
\end{theorem}

The theorem is proved in Section \ref{proofs}.

\medskip

Of course, the bijection does not necessarily restrict to a bijection from $\p L(\mu,\lambda) = \p B(\mu,\lambda) \times \p W(\mu,\lambda)$ to $\p R(\mu,\lambda) = \bigcup_{\nu} \p B(\nu,\lambda)$, and does not immediately prove Theorem \ref{thm1}. Once we insert a variable from $\p W(\mu,\lambda)$ into a bicolored tableau in $\p B(\mu,\lambda)$, the resulting tableau can add an outer corner of $\mu$ which is not in $[\lambda]$, or it can return a bicolored tableau in $\p B(\nu)$, $\nu \subseteq \lambda$, which is not in $\p B(\nu, \lambda)$. For instance, the last example produced a tableau in $\p B(432) \setminus \p B(432,765521)$.

\medskip

If $\psi_\mu(T,z) \in \p B(\nu)$ is not in $\p R(\mu,\lambda)$, we can remove the entry in the unique cell in $[\nu/\mu]$ and obtain a new variable $z'$ and a tableau $T'$ of shape $\mu$. Compute  $\psi_\mu(T',z') \in \p B({\nu}')$. If it is in $\p R(\mu,\lambda)$, terminate the procedure, otherwise remove the entry in the unique cell in $[{\nu}'/\mu]$ and obtain a new variable $z''$ and a tableau $T''$ of shape $\mu$. Continue until the computed tableau is in $\p R(\mu,\lambda)$; the procedure returns this tableau as the result. We call this the \emph{repeated insertion process}.

\begin{example}
 Take $\mu = 431$, $\lambda = 765521$,
 $$T =     \ytableausetup{boxsize=1.25em}
 \ytableausetup{aligntableaux=center}
 \ytableaushort
{0{\rd 0}{\rd 0}1,01{\rd 1},0}
\in \p B(431,765521)$$
 and $z = y_1 \in \p W(431,765521)$. We already computed
 $$\psi_{431}(T,z) = \ytableausetup{boxsize=1.25em}
 \ytableausetup{aligntableaux=center}
 \ytableaushort
{{\rd 0}0{\rd 0}1,0{\rd 0}1,0 {\rd 2}} \in \p B(432) \setminus \p B(432,765521).$$
 Remove the red $2$ from position $(3,2)$, and insert $z' = y_4$ into the tableau
 $$T' = \ytableausetup{boxsize=1.25em}
 \ytableausetup{aligntableaux=center}
 \ytableaushort
{{\rd 0}0{\rd 0}1,0{\rd 0}1,0 }.$$
The result is
$$\psi_{431}(T',z') = \ytableausetup{boxsize=1.25em}
 \ytableausetup{aligntableaux=center}
 \ytableaushort
{{\rd 0}0{\rd 0}{\rd 0}1,0{\rd 0}1,{\rd 0}},$$
which is an element of $\p B(531,765521) \subseteq \p R(431,765521)$. Therefore the procedure terminates and returns
$$\ytableausetup{boxsize=1.25em}
 \ytableausetup{aligntableaux=center}
 \ytableaushort
{{\rd 0}0{\rd 0}{\rd 0}1,0{\rd 0}1,{\rd 0}}.$$
\end{example}

\begin{theorem} \label{repeated}
 The repeated insertion process described above always terminates and is a weight-preserving bijection
 $$\Psi_{\mu,\lambda} \colon \p B(\mu,\lambda) \times \p W(\mu,\lambda) \longrightarrow \bigcup_{\nu} \p B(\nu,\lambda),$$
 where the union is over all partitions $\nu$ which cover $\mu$ and are contained in $\lambda$.
\end{theorem}

The last theorem proves \eqref{eq4} and hence the hook-length formula for skew shapes, equation \eqref{eq}.

\medskip

The proof of Theorem \ref{repeated} is also presented in Section \ref{proofs}.

\section{Proofs} \label{proofs}

\subsection*{Proof of Lemma \ref{w}}

Recall that we have $x_k = \lambda_k - k + \frac 1 2$ and $y_k = \lambda'_k - k + \frac 1 2$. Define also $x'_i = \mu_i - i + \frac 1 2$ and $y'_j = \mu'_j - j + \frac 1 2$. We are interested in the expression
$$|\lambda| - |\mu| = x_1 + \cdots + x_{r(\lambda)} + y_1 + \cdots + y_{r(\lambda)} - x'_1 - \cdots - x'_{r(\mu)} - y'_1 - \cdots - y'_{r(\mu)}.$$

Let us study the sequence of cells
\begin{multline*}
\ldots \to (k-2,\lambda_k - 2)\to (k-2,\lambda_k - 1) \to (k-1,\lambda_k - 1) \to (k-1,\lambda_k) \\ \to (k,\lambda_k) \to (k,\lambda_k+1) \to (k+1,\lambda_k+1) \to (k+1,\lambda_k+2) \to (k+2,\lambda_k+2) \to \ldots
\end{multline*}

\medskip
We are interested in $u$, the first cell in the sequence that has positive coordinates and is not in $[\mu]$. One option is that $u = (k+l,\lambda_k+l+1)$ for some $l \in \Z$. In that case, either $(k+l,\lambda_k+l) \in [\mu]$ or $\lambda_k + l = 0$. In both cases, $\mu_{k+l} = \lambda_k+l$. So we have $\lambda_k - k = \mu_i - i$ and $x_k = x'_i$ for $i = k+l$. The converse also holds: if $x_k = x'_i$ for some $i$, then $ \lambda_k - k = \mu_i - i$ and $u = (i,\lambda_k+i-k+1)$. The other option is that $u = (k+l,\lambda_k+l)$ for some $l$. Now either $(k+l-1,\lambda_k+l) \in [\mu]$ or $k+l-1 = 0$. In both cases, $\mu'_{\lambda_k+l} = k+l-1$. So $\mu'_j = k+j-\lambda_k-1$, $\lambda_k - k + 1 = j - \mu'_j$ and $x_k = -y'_j$ for $j = \lambda_k+l$. Conversely, if $x_k = -y'_j$ for some $j$, then $\lambda_k - k + 1 = j - \mu'_j$ and $u = (k+j-\lambda_k,j)$.

\medskip

We have seen that either $x_k = x'_i$ or $x_k = -y'_j$ for some (unique) $i$ or $j$, but not both. Furthermore, if $k \leq r(\lambda)$, then $x_k \geq 0$, so $x_k = x'_i$ implies $i \leq r(\mu)$ and $x_k = -y'_j$ implies $j > r(\mu)$. Similarly, if we study the sequence of cells
\begin{multline*}
\ldots \to (k-2,\lambda_k - 2)\to (k-1,\lambda_k - 2) \to (k-1,\lambda_k - 1) \to (k,\lambda_k-1) \\ \to (k,\lambda_k) \to (k+1,\lambda_k) \to (k+1,\lambda_k+1) \to (k+2,\lambda_k+1) \to (k+2,\lambda_k+2) \to \ldots,
\end{multline*}
then we see that either $y_k = y'_j$ or $y_k = -x'_i$ for some (unique) $j$ or $i$, but not both. Furthermore, if $k \leq r(\lambda)$, then $y_k \geq 0$, so $y_k = y'_j$ implies $j \leq r(\mu)$ and $y_k = -x'_i$ implies $i > r(\mu)$.

\medskip

Recall that $\lambda$ and $\mu$ are arbitrary partitions (i.e., we do not assume that $\mu \subseteq \lambda$). So we can switch the roles of $\lambda$ and $\mu$ in the above computations, and express $x'_i$ and $y'_j$ in terms of $x_k$'s and $y_k$'s.

\medskip

After we express $x'_i$'s and $y'_j$'s in terms of $x_k$'s and $y_k$'s, the coefficient of $x_k$ in $|\lambda| - |\mu|$ is:
\begin{itemize}
 \item $1$ if $k \leq r(\lambda)$ and there is no $i$ so that $x_k = x'_i$
 \item $0$ if $k \leq r(\lambda)$ and $x_k = x'_i$ for some $i$ (necessarily $i \leq r(\mu)$)
 \item $1$ if $k > r(\lambda)$ and $x_k = - y'_i$ for some $i$ (necessarily $i \leq r(\mu)$); equivalently, if there is no $i$ so that $x_k = x'_i$
 \item $0$ if $k > r(\lambda)$ and there is no $i$ so that $x_k = -y'_i$; equivalently, if $x_k = x'_i$ for some $i$ (necessarily $i > r(\mu)$)
\end{itemize}

To summarize, $x_k$ appears as a term in $|\lambda| - |\mu|$ if and only if there is no $i$ so that $x_k = x'_i$, which is equivalent to $\lambda_k - k = \mu_i - i$. Similarly, we see that $y_k$ appears as a term in $|\lambda| - |\mu|$ if and only if there is no $j$ so that $y_k = y'_j$, which is equivalent to $\lambda'_k - k = \mu'_j - j$. This finishes the proof of Lemma \ref{w}.

\subsection*{The insertion process and the proof of Theorem \ref{insert}}

In this subsection, we prove the technical properties of the insertion process $\psi_\mu$, including Theorem \ref{insert}.

\medskip

Say that we are at a certain step of the insertion process, and that a black $r$ was just bumped from position $(i',j-1)$ (analysis for a red $r$ is analogous). The algorithm says that we should find the largest possible $i$ so that we can write $r + i' - i$ in position $(i,j)$ while keeping the column $j$ weakly increasing. Note that since the sequence $(S_{ij})_{i=0}^{\mu'_j+1}$ (where we interpret $S_{0,j}$ as $0$ and $S_{\mu'_j+1,j}$ as $\infty$) is weakly increasing, $(S_{ij} + i - i')_{i=0}^{\mu'_j+1}$ is strictly increasing, and we have to find the largest possible $i$ so that if we replace the $i$-entry of the sequence with $r$, we still have a strictly increasing sequence.

\medskip

It is clear that if $r = S_{ij} + i'' - i'$ for some $i''$, we have just one choice for $i$, so we pick $i = i''$, and if $ S_{i''-1,j} + i'' - 1 - i' < r < S_{i'',j} + i'' - i'$ for some $i''$, we have two choices for $i$, $i''-1$ and $i''$, and we pick the larger one, $i = i''$. So $i$ is indeed well defined, and $S_{ij} + i - i' \geq r$. Furthermore, if $\mu'_{j-1} = \mu'_j$, then $i' \leq \mu_j'$, $r \leq S_{\mu'_j,j} \leq S_{\mu'_j,j} + \mu_j' - i'$ and the chosen $i$ is not $\mu'_j+1$. That means that we only add a cell $(\mu_j'+1,j)$ to the tableau (and terminate the process) when $\mu'_{j-1} > \mu'_j$, i.e.~when $(\mu_j'+1,j)$ is an outer corner of $\mu$.

\medskip

Since $r \leq S_{i',j} < S_{i'',j} + i'' - i'$ for $i'' > i'$, we always have $i \leq i'$. In other words, when the process moves by one to the right, it cannot go down (and when the process moves down by one, it cannot go to the right). Furthermore, we notice that the number bumped by $r$ is $S_{ij} \geq r + i' - i \geq r$, i.e., a number is never bumped by a strictly larger number (also when bumping a red number).

\medskip

That means that the new entry in position $(i,j)$ is still less than or equal to $S_{i,j+1}$ (which we take to be $\infty$ if $(i,j+1)$ is not in the diagram). Furthermore, since $S_{i,j-1} \leq S_{i',j-1} = r$, and since the new entry in position $(i,j)$ is at least as large as $r$, the new entry in position $(i,j)$ is still greater than or equal to $S_{i,j-1}$. In other words, the new tableau still has weakly increasing rows and columns.

\medskip

In order to prove that the process terminates, it is enough to prove that a certain (integer) quantity with an upper bound increases at each step. We claim that such a quantity is $i + j + s$, where $(i,j)$ is the current position and $s$ is the number getting bumped. It is clear that the quantity is bounded by $\ell(\mu) + \mu_1 + \max T_{ij}$. Also, in the same notation as before, the quantity was $i' + (j-1) + r$ in the previous step and is $i + j + S_{ij} \geq i' + j + r$ in the current step, so it increases by at least $1$ (a similar proof can be written for the case when we are bumping a red number).

\medskip

It remains to construct the inverse of the process. Start with a bicolored tableau $T \in \p B(\nu)$, where $\mu \lessdot \nu$. Assume that after some number of steps, we have $i$, $j$, $S$ and $w$; in the beginning, $(i,j)$ is the unique cell in $[\nu/\mu]$, $S$ is $T$ with the entry in $(i,j)$ removed, and $w$ the variable corresponding to that entry ($x_{i+T_{ij}}$ if the entry is black and $y_{j+T_{ij}}$ if it is red). If $w = x_k$, we decrease $j$ by $1$, i.e., we move by one to the left, and, again, we find the largest possible $i$ so that we can replace $S_{ij}$ by a black $k - i$ in position $(i,j)$ and still have a weakly increasing column with non-negative integers. If $w = y_k$, the process is analogous. It is easy to prove that this process is well defined, terminates and is the inverse of $\psi_\mu$.

\medskip

The fact that the $\psi_\mu$ is weight preserving is obvious, as the weight (including the bumped variable) is preserved at each bump.

\subsection*{The repeated insertion process and the proof of Theorem \ref{repeated}}

In this subsection, we prove the technical properties of the repeated insertion process $\Psi_{\mu,\lambda}$.

\medskip

Suppose we have finite sets $A$ and $B$ and a bijection $f \colon A \to B$. Furthermore, suppose we have subsets $X \subseteq A$ and $Y \subseteq B$ and a bijection $g \colon B \setminus Y \to A \setminus X$. For $x \in X$, let $m$ be the smallest (and only) non-negative integer such that $f \circ (g \circ f)^m(x) \in Y$, and define $h(x) = f \circ (g \circ f)^m(x)$. It is easy to see that $h \colon X \to Y$ is a well-defined bijection. Furthermore, if we have weights $w_A \colon A \to K$, $w_B \colon B \to K$ and $f$ and $g$ are weight preserving (i.e., $w_B(f(a)) = w_A(a)$ for all $a \in A$ and $w_A(g(b)) = w_B(b)$ for all $b \in B \setminus Y$), then $h$ is also weight preserving, i.e., $w_B(h(x)) = w_A(x)$ for all $x \in X$. See for example \cite[\S 2.6]{ec1}, where the process is called \emph{sieve equivalence}. It is also similar to the well-known Garsia-Milne \emph{involution principle} \cite{gm}. Note that the involution principle was used, for example, in the proof of the hook-length formula \cite{krat1} and in the first bijective proof of the hook-content formula in \cite{rw}.

\medskip

Define $\p W(\lambda) = \{x_1,\ldots,x_{\ell(\lambda)},y_1,\ldots,y_{\lambda_1}\}$, the set of all possible variables that can appear in $\p L(\mu,\lambda)$ and $\p R(\mu,\lambda)$.

\medskip

In our context, we define the following:
\begin{itemize}
 \item $A$ is the (finite) product $\p B(\mu,\lambda) \times \p W(\lambda)$,
 \item $B$ is the image $\psi_\mu(A) \subseteq \bigcup_\nu \p B(\nu)$,
 \item $f$ is the (bijective) map $\psi_\mu$, insertion of a variable into a bicolored tableau of shape $\mu$,
 \item $g$ is the map $\varphi_\mu$, which takes a bicolored tableau $T \in \p B(\nu)$, $\mu \lessdot \nu$, and produces the pair $(S,z)$, where $S$ is $T$ with the unique entry in position $(i,j) \in [\nu/\mu]$ removed, and $z$ is the variable corresponding to the removed entry ($x_{i+T_{ij}}$ if the removed entry is black, and $y_{j+T_{ij}}$ if it is red),
 \item $X$ is $\p L(\mu,\lambda)$, i.e., $\p B(\mu,\lambda) \times \p W(\mu,\lambda)$,
 \item $Y$ is $\p R(\mu,\lambda)$, i.e., $\bigcup_\nu \p B(\nu,\lambda)$.
\end{itemize}

Note that $g$ is not bijective. For example, the following tableaux all give the same tableau of shape $431$ and the variable $x_4$ upon removal of the entry in cell $(1,5)$ (respectively, $(2,4)$, $(3,2)$, $(4,1)$).

\medskip

 \begin{center}
   \ytableausetup{boxsize=1.25em}
 \ytableausetup{aligntableaux=top}
 \ytableaushort
{0{\rd 0}003,01{\rd 1},0}
\qquad\qquad
   \ytableausetup{boxsize=1.25em}
 \ytableausetup{aligntableaux=top}
 \ytableaushort
{0{\rd 0}00,01{\rd 1}2,0}
\qquad\qquad
   \ytableausetup{boxsize=1.25em}
 \ytableausetup{aligntableaux=top}
 \ytableaushort
{0{\rd 0}00,01{\rd 1},01}
\qquad\qquad
   \ytableausetup{boxsize=1.25em}
 \ytableausetup{aligntableaux=top}
 \ytableaushort
{0{\rd 0}00,01{\rd 1},0,0}
\end{center}

\medskip

However, for the sieve equivalence to work, it is enough that $g$ is bijective when restricted to $B \setminus Y$.

\begin{lemma}
 The restriction of $g$ to $B \setminus Y$ is injective, and its image is $A \setminus X$. Furthermore, $f$ and $g$ are weight preserving, and hence $h$ is weight preserving. In other words, $\Psi_{\mu,\lambda}$ is a well-defined weight-preserving bijection.
\end{lemma}
\begin{proof}
 We proved in the previous subsection that if we insert a variable into a tableau of $\mu$, the entries in $[\mu]$ are smaller than or equal to their previous values. So after inserting a variable from $\p W(\lambda)$ into $T \in \p B(\mu,\lambda)$ and removing the new corner, we again get a tableau in $\p B(\mu,\lambda)$. In other words, $\varphi_\mu \circ \psi_\mu (\p B(\mu,\lambda) \times \p W(\lambda)) \subseteq \p B(\mu,\lambda) \times \p W(\lambda)$, so indeed $g \colon B \to A$.\\
 We need to prove that $g$ restricted to $B \setminus Y$ maps to $A \setminus X$. In other words, we have to prove that if $\psi_\mu(T,z) \notin \bigcup_\nu \p B(\nu,\lambda)$, then $\varphi_\mu \circ \psi_\mu(T,z) \notin \p B(\mu,\lambda) \times \p W(\mu,\lambda)$.\\
The assumption is that after we insert $z$ into $T$, the result $S \in \p B(\nu)$ is not in $\p B(\nu,\lambda)$. In other words, the entry of $S$ in position $(i,j) \in [\nu/\mu]$ is too large, i.e., we have $j + S_{ij} > \lambda_{i + S_{ij}}$ (this includes the case when $\nu \not\subseteq \lambda$). Assume that the entry in position $(i,j)$ is black (the analysis for a red entry is analogous) and that it represents the variable $x_k$ (so $S_{ij} +i = k$). We have $j + k - i > \lambda_k$, i.e.~$j-1-i \geq \lambda_k - k$.\\
The variable $x_k$ was bumped from the previous column, say from position $(i',j-1)$, where $i' \geq i$. Since $(i',j-1) \in [\mu]$ and $(i,j)$ is an outer corner of $\mu$, we have $\mu_i = \mu_{i+1} = \ldots = \mu_{i'} = j-1$, $\mu'_{j-1} \geq i'$ and $\mu'_j = i-1$. We want to prove that $x_k$ is not an element of $\p W(\mu,\lambda)$.\\
Before getting bumped, the entry in $(i',j-1)$ was $k-i'$ (so that it represented the same variable $x_k$), and it satisfied $j-1+ k-i' \leq \lambda_{i'+k-i'} = \lambda_k$, see \eqref{crit}. In other words, we have $j-1-i'\leq \lambda_k - k \leq j - 1 - i$. But that means that for $i''=j-1+k-\lambda_k$, $i \leq i'' \leq i'$, we have $\mu_{i''} - i'' = (j-1) - (j-1+k-\lambda_k) = \lambda_k - k$, so $x_k \notin \p W(\mu,\lambda)$.\\
We now need to prove that $g|_{B \setminus Y}$ is injective. In other words, we need to prove that if $\psi_\mu(T,z), \psi_\mu(T',z') \notin \bigcup_\nu \p B(\nu,\lambda)$ and $\varphi_\mu \circ \psi_\mu(T,z) = \varphi_\mu \circ \psi_\mu(T',z')$, then $T = T'$ and $z = z'$.\\
The assumption is saying that after we bump $z$ into $T$ and $z'$ into $T'$, we get bicolored tableaux which are of different shapes $\nu$ and $\nu'$, but the entries in the unique cells $(i,j)$ and $(i',j')$ of $[\nu/\mu]$ and $[\nu'/\mu]$ represent the same variable (without loss of generality, $x_k$ for some $k$), and after deleting these entries, we get the same bicolored tableau of shape $\mu$. Without loss of generality, $j < j'$.\\
Furthermore, the variable $x_k$ is too big for either position $(i,j)$ or $(i',j')$. We saw earlier in this proof that we must have $\lambda_k < j+k-i$ and $\lambda_k \geq j'+k-i''-1$, where the variable $x_k$ was bumped from position $(i'',j'-1)$ to $(i',j')$ in the insertion of $z'$ into $T'$. However, $j \leq j'-1$ and $i \geq \mu_j' + 1 \geq \mu'_{j'-1} + 1\geq i''+1$, so $j+k-i \leq j'-1+k-i''-1 < j'+k-i''-1$, which is a contradiction.\\
Finally, we have to prove that  $g \colon B \setminus Y  \to A \setminus X$ is surjective. In other words, we have to prove that if $(T,z) \in \p B(\mu,\lambda) \times (\p W(\lambda) \setminus \p W(\mu,\lambda))$, then there exists $(T',z')$ such that $\psi_\mu(T',z') \notin \bigcup_\nu \p B(\nu,\lambda)$ and $\varphi_\mu \circ \psi_\mu(T',z') = (T,z)$.\\
We assume that $z = x_k \in \p W(\lambda) \setminus \p W(\mu,\lambda)$ for some $k$; the analysis for $z = y_k$ is very similar and is left as an exercise for the reader.\\
For some $i$, we have $\lambda_k - k = \mu_i - i$. Write $j = \mu_i + 1$ and $i' = \mu'_j+1 \leq i$. Place a black $k-i'$ in position $(i',j)$ and denote the resulting tableau $S$. We claim that $S$ is a bicolored tableau and that $(T',z') = \psi_\mu^{-1}(S)$ satisfies the required properties.\\
For $l > k - i$, $j-1 + l \leq \lambda_{i+l}$ is not satisfied; indeed, in this case $j-1 + l > j-1 + k - i = \lambda_k \geq \lambda_{i+l}$. In other words, $S_{i,j-1} = T_{i,j-1} \leq k-i$ and so $S_{i',j-1} = T_{i',j-1} \leq k-i$. That means that writing $k-i' \geq k-i$ in position $(i',j)$ does not create a decrease in row $i'$. Also, for $l > k - i'$, $j+l \leq \lambda_{i'-1+l}$ is not satisfied, as $j + l > j + k - i' \geq j + k - i > \lambda_k \geq \lambda_{i'-1+l}$. So $S_{i'-1,j} = T_{i'-1,j} \leq k - i'$ and writing $k-i'$ in position $(i',j)$ does not create a decrease in column $j$. We have proved that $S$ is indeed a bicolored tableau. Let us denote its shape by $\nu$, so $(i',j)$ is the only cell in $[\nu/\mu]$. Write $(T',z') = \psi_\mu^{-1}(S)$.\\
We claim that $S = \psi_\mu(T',z')$ is not in $\p B(\nu,\lambda)$, i.e.~that $k-i'$ is too large for position $(i',j)$. Indeed, $j+k-i' \geq j+k-i > \lambda_k = \lambda_{i'+k-i'}$.\\
We also claim that $T'$ is in $\p B(\mu,\lambda)$. When we start the inverse insertion process, we put the variable $x_k$ into column $j-1$ of $S$. However, we saw that $S_{i',j-1} \leq \ldots \leq S_{i,j-1} \leq k - i$, and since we write the variable $x_k$ in position $(i'',j-1)$, where $i' \leq i'' \leq i$, as a black $k - i''$, we must have $i'' = i$. We therefore have $k-i$ in position $(i,j-1)$, which is, by $j-1+k-i = \lambda_k$, not too large for the result to not be in $\p B(\mu,\lambda)$. Continuing with the reverse insertion process does not change that fact: if $x_k$ comes from a black $k-i$ in position $(i,j) \in [\mu]$, $j + k - i \leq \lambda_k$, and lands in $(i',j-1)$, $i' \geq i$, as a black $k-i'$, then $j-1+k-i' < j+k-i \leq \lambda_k$ (and a similar proof for $y_k$). Furthermore, the variable $z'$ will obviously be in $\p W(\lambda)$.
\end{proof}

The lemma proves Theorem \ref{repeated}.

\section{Final remarks} \label{final}

\subsection*{Comparison of Naruse's formula with others} Naruse's formula seems better for many applications, e.g.~asymptotics; see for example \cite[Section 9]{mpp} and \cite{mppasymp}. This paper presents another advantage: it has a natural bijective proof.

\subsection*{Connection to Ikeda-Naruse's formula} It was pointed out by Morales and Panova (personal communication, see also \cite{mpp3}) that in \cite[equation (5.2)]{IN1}, Ikeda and Naruse proved algebraically that for a skew shape $\lambda/\mu$ that fits inside a $d \times (n-d)$ box,
$$(F(\lambda/1) - F(\mu/1)) F(\lambda/\mu) = \sum_{\mu \lessdot \nu \subseteq \lambda} F(\lambda/\nu),$$
where
$$F(\lambda/\mu) = \sum_{D \in \p E(\lambda/\mu)} \prod_{(i,j) \in D} (z_{\lambda_i + d - i + 1} - z_{d+j-\lambda'_j}).$$
In particular,
$$F(\lambda/1) = \sum_{i = 1}^{r(\lambda)} (z_{\lambda_i + d - i + 1} - z_{d+j-\lambda'_j}).$$
For $n$ and $d$ fixed, we introduce variables $x_{\lambda_i-i} = z_{\lambda_i+d-i+1}$, $1 \leq i \leq d$, and $y_{\lambda'_j-j} = -z_{d+j-\lambda'_j}$, $1 \leq j \leq n-d$ (we always have $\lambda_i+d-i+1 \neq d+j-\lambda'_j$ since the difference is the hook length of the cell $(i,j)$), and get precisely \eqref{eqmain}.

\subsection*{Bijective proof of Monk's formula}

It was pointed out by Sara Billey that formula \eqref{eqmain} is similar to Monk's formula for Schubert polynomials. Indeed, the double Schubert polynomial of a permutation $w$ is the sum of $\prod_{(i,j) \in D} (x_i +y_j)$ over all RC-graphs $D$ for $w$. It would be interesting to see if there is a connection between our bijection and the bijective proof of Monk's formula from \cite{bergeronbilley}.

\subsection*{Skew shifted shapes}

An obvious question is how to adapt the bijection to prove the version of Naruse's hook-length formula for skew \emph{shifted} shapes. While one might expect that a version of such a bijection would be much more complicated than the one presented here, it turns out that the proof can be adapted without major difficulties. See \cite{shifted}.

\section*{Acknowledgments}

The author would like to thank Alejandro Morales, Igor Pak and Greta Panova for telling him about the problem and the interesting discussions that followed. Many thanks also to Sara Billey for reading an early draft of the paper so carefully and for giving numerous useful comments, to Darij Grinberg for a number of wonderful suggestions, in particular for the one leading to a simplification of the proof of Theorem \ref{repeated}, and to Graham Gordon for finding a typo in \eqref{oo}.

\bibliographystyle{siam}
\def\cprime{$'$}

\end{document}